\documentclass[a4paper,12pt]{amsart}
\title[Semiring isomorphisms between rational function semifields]{Semiring isomorphisms between rational function semifields of tropical curves induce isomorphisms between tropical curves}

\author{Song JuAe}
\address{Tokyo Metropolitan University 1-1 Minami-Ohsawa, Hachioji, Tokyo, 192-0397, Japan.}
\email{song-juae@ed.tmu.ac.jp}


\usepackage{amsmath,amssymb,amscd}
\usepackage{amsthm}
\usepackage{color}
\usepackage{subfigure}
\usepackage{comment}
\usepackage{mathrsfs}

\newtheorem{dfn}{Definition}[section]
\newtheorem{thm}[dfn]{Theorem}
\newtheorem{prop}[dfn]{Proposition}
\newtheorem{cor}[dfn]{Corollary}
\newtheorem{lemma}[dfn]{Lemma}
\newtheorem{rem}[dfn]{Remark}

\def\Gamma{\varGamma}

\begin{document}

\maketitle

\begin{abstract}
We prove that a semiring isomorphism between the rational function semifields of two tropical curves induces an expansive map between those tropical curves.
This semiring isomorphism and the expansive map respect zeros and poles of rational functions with their degrees.
As a corollary, we show that the automorphism group of a tropical curve is isomorphic to the $\boldsymbol{T}$-algebra automorphism group of its rational function semifield, where $\boldsymbol{T} := (\boldsymbol{R} \cup \{ -\infty \}, \operatorname{max}, +)$ is the tropical semifield.
Finally, we describe all semiring automorphisms of rational function semifields of all tropical curves.
\end{abstract}

\section{Introduction}
	\label{section1}

\newcounter{num}
\setcounter{num}{1}
The category of nonsingular projective curves and dominant morphisms is equivalent to the category of function fields of dimension one over $k$ and $k$-homomorphisms, where $k$ is an algebraically closed field (see \cite[Corollary. 6.12. in Chapter \Roman{num}]{Hartshorne}).

Is there a tropical analogue of this fact?
We factorize this question into the following three subquestions.

$(1)$ Is the rational function semifield $\operatorname{Rat}(\Gamma)$ of a tropical curve $\Gamma$ a finitely generated semifield over the tropical semifield $\boldsymbol{T} := (\boldsymbol{R} \cup \{- \infty \}, \operatorname{max}, +)$?
If so, does $\operatorname{Rat}(\Gamma)$ have dimension one over $\boldsymbol{T}$ (for the definition of dimension, see \cite{Joo=Mincheva})?

$(2)$ What kind of semifields over $\boldsymbol{T}$ determines a tropical curve?
In particular, if $(1)$ holds, does a finitely generated semifield over $\boldsymbol{T}$ of dimension one determine a tropical curve?
How do we construct it?

$(3)$ For two tropical curves $\Gamma_1$ and $\Gamma_2$, does a $\boldsymbol{T}$-algebra isomorphism $\operatorname{Rat}(\Gamma_1) \to \operatorname{Rat}(\Gamma_2)$ induce an isomorphim (i.e., a finite harmonic morphism of degree one) $\Gamma_1 \to \Gamma_2$?

In \cite{JuAe3}, the author considered the question $(1)$.
The main theorem of \cite{JuAe3} states that $\operatorname{Rat}(\Gamma)$ is finitely generated over $\boldsymbol{T}$ as a tropical semifield.
The rest of $(1)$ and the question $(2)$ are still open.
The question $(3)$ is our main topic; the answer is YES.
The following is our main theorem:

\begin{thm}
	\label{thm1}
Let $\Gamma_1, \Gamma_2$ be tropical curves.
Let $\psi : \operatorname{Rat}(\Gamma_1) \to \operatorname{Rat}(\Gamma_2)$ be a semiring isomorphism between their rational function semifields.
Then $\psi(1) \in \boldsymbol{R}_{>0}$ and $\psi$ induces a $\psi(1)$-expansive map $\varphi : \Gamma_1 \to \Gamma_2$.
In particular, when $\psi$ is a $\boldsymbol{T}$-algebra isomorphism, $\psi(1) = 1$ and $\varphi$ is a finite harmonic morphism of degree one.
\end{thm}

Here, for $r \in \boldsymbol{R}_{>0}$, an \textit{$r$-expansive map} between tropical curves $\varphi : \Gamma_1 \to \Gamma_2$ means a continuous bijection which induces a metric isomorphism between $\Gamma_1 \setminus \Gamma_{1 \infty}$ and $\Gamma_2 \setminus \Gamma_{2 \infty}$ with $r$ as its expansion factor, i.e., $r \cdot \operatorname{dist}(x, y) = \operatorname{dist}(\varphi(x), \varphi(y))$ for any $x, y \in \Gamma_1 \setminus \Gamma_{1 \infty}$, where $\Gamma_{i \infty}$ denotes the set of all points at infinity of $\Gamma_i$, $\operatorname{dist}(x, y)$ denotes the distance between $x$ and $y$, $\cdot$ denotes the usual multiplication of real numbers and we define $r \cdot \infty = \infty$.

Theorem \ref{thm1} is an extension of semiring automorphism $\boldsymbol{T} \to \boldsymbol{T}$ case, that is, for a semiring automorphism $\psi : \boldsymbol{T} \to \boldsymbol{T}$, $\psi(1) \in \boldsymbol{R}_{>0}$ and $\psi$ is the \textit{expansive map} $\boldsymbol{T} \to \boldsymbol{T}; t \mapsto \psi(1) \cdot t$, where we define $r \cdot (-\infty) = -\infty$ for any $r \in \boldsymbol{R}_{>0}$ (Proposition \ref{prop1}).
One of the most important consequences of Theorem \ref{thm1} is that the $\boldsymbol{T}$-algebra structure of the rational function semifield of a tropical curve perfectly determines (not only the topological structure but) the metric structure of the tropical curve.

Theorem \ref{thm1} has the following corollary:

\begin{cor}
	\label{cor1}
The following groupoids $\mathscr{C}, \mathscr{D}$ are isomorphic.

$(1)$ The class $\operatorname{Ob}(\mathscr{C})$ of objects  of $\mathscr{C}$ is the tropical curves.

For $\Gamma_1, \Gamma_2 \in \operatorname{Ob}(\mathscr{C})$, the set $\operatorname{Hom}_{\mathscr{C}}(\Gamma_1, \Gamma_2)$ of morphisms from $\Gamma_1$ to $\Gamma_2$ consists of the semiring isomorphisms $\operatorname{Rat}(\Gamma_1) \to \operatorname{Rat}(\Gamma_2)$.

$(2)$ The class $\operatorname{Ob}(\mathscr{D})$ of objects of $\mathscr{D}$ is the tropical curves.

For $\Gamma_1, \Gamma_2 \in \operatorname{Ob}(\mathscr{D})$, the set $\operatorname{Hom}_{\mathscr{D}}(\Gamma_1, \Gamma_2)$ of morphisms from $\Gamma_1$ to $\Gamma_2$ consists of the $r$-expansive maps $\Gamma_1 \to \Gamma_2$ for all $r \in \boldsymbol{R}_{>0}$.
\end{cor}

In particular, considering the case that $\Gamma = \Gamma_1 = \Gamma_2$ and $\psi$ is a $\boldsymbol{T}$-algebra automorphism of $\operatorname{Rat}(\Gamma)$, we have the following corollary:

\begin{cor}
	\label{cor2}
The automorphism group of a tropical curve $\Gamma$ is isomorphic to the $\boldsymbol{T}$-algebra automorphism group of $\operatorname{Rat}(\Gamma)$.
\end{cor}

Since $\varphi$ is a $\psi(1)$-expansive map $\Gamma_1 \to \Gamma_2$, for any $f \in \operatorname{Rat}(\Gamma_1)$, slopes of $f$ and $\psi(f)$ on intervals corresponding by $\varphi$ coincide.
Hence we obtain the following corollary:

\begin{cor}
	\label{cor3}
Let $\Gamma_1, \Gamma_2$ be tropical curves.
A semiring isormorphism $\psi: \operatorname{Rat}(\Gamma_1) \to \operatorname{Rat}(\Gamma_2)$ and the induced $\psi(1)$-expansive map $\varphi: \Gamma_1 \to \Gamma_2$ map zeros (resp. poles) to zeros (resp. poles) and preserve their degrees.
\end{cor}

Note that we can restrict ourselves to $\boldsymbol{T}$-algebra isomorphism setting.
In this setting, $\psi(1)$ is always one, and thus we need not deal with $r$-expansive maps.

Here we note the work in \cite{Jun} related to the question $(2)$ above.
In \cite{Jun}, Jun deals with three kinds of valuations of semirings, \textit{classical}, \textit{strict} and \textit{hyperfield valuations}, and has constructed the abstract curve associated to the rational function semifield over $(\boldsymbol{Q} \cup \{ -\infty \}, \operatorname{max}, +)$ via strict valuations.

The rest of this paper is organized as follows.
In Section \ref{section2}, we give the definitions of semirings and algebras, tropical curves, rational functions and chip-firing moves on tropical curves, and finite harmonic morphisms between tropical curves.
Section \ref{section3} contains proofs of all the assertions above.
In that section, we also describe all semiring automorphisms of rational function semifields of all tropical curves; for all tropical curves $\Gamma$ except star-shaped tropical curves consisting of a finite number (at least one) of $[0, \infty]$, the semiring automorphism group of $\operatorname{Rat}(\Gamma)$ coincides with the $\boldsymbol{T}$-algebra automorphism group of $\operatorname{Rat}(\Gamma)$ (Lemma \ref{lem7} and Corollary \ref{cor6}).

\section*{Acknowledgements}
The author thanks my supervisor Masanori Kobayashi, Yuki Kageyama, Yasuhito Nakajima, Ken Sumi, and Daichi Miura for helpful comments.
This work was supported by JSPS KAKENHI Grant Number 20J11910.

\section{Preliminaries}
	\label{section2}

In this section, we recall several definitions which we need later.
We refer to \cite{Golan} (resp. \cite{Maclagan=Sturmfels}) for an introduction to the theory of semirings (resp. tropical geometry) and employ definitions in \cite{Jun} (resp. \cite{JuAe}) related to semirings (resp. tropical curves).

\subsection{Semirings and algebras}
	\label{subsection2.1}

In this paper, a \textit{semiring} is a commutative semiring with the absorbing neutral element $0$ for addition and the identity $1$ for multiplication such that $0 \not= 1$.
If every nonzero element of a semiring $S$ is multiplicatively invertible, then $S$ is called a \textit{semifield}.
A semiring $S$ is \textit{additively idempotent} if $x + x = x$ for any $x \in S$.
An additively idempotent semiring $S$ has a natural partial order, i.e., for $x, y \in S$, $x \ge y$ if and only if $x + y = x$.

A map $\varphi : S_1 \to S_2$ between semirings is a \textit{semiring homomorphism} if for any $x, y \in S_1$,
\[
\varphi(x + y) = \varphi(x) + \varphi(y), \	\varphi(x \cdot y) = \varphi(x) \cdot \varphi(y), \	\varphi(0) = 0, \	\text{and}\	\varphi(1) = 1.
\]
A semiring homomorphism $\varphi : S_1 \to S_2$ is a \textit{semiring isomorphism} if $\varphi$ is bijective.
A \textit{semiring automorphism} of $S$ is a semiring isomorphism $S \to S$.

Given a semiring homomorphism $\varphi : S_1 \to S_2$, we call the pair $(S_2, \varphi)$ (for short, $S_2$) a \textit{$S_1$-algebra}.
For a semiring $S_1$, a map $\psi : (S_2, \varphi) \to (S_2^{\prime}, \varphi^{\prime})$ between $S_1$-algebras is a \textit{$S_1$-algebra homomorphism} if $\psi$ is a semiring homomorphism and $\varphi^{\prime} = \psi \circ \varphi$.
When there is no confusion, we write $\psi : S_2 \to S_2^{\prime}$ simply.
A bijective $S_1$-algebra homomorphism $S_2 \to S_2^{\prime}$ is a \textit{$S_1$-algebra isomorphism}.
In addition, if $S_2 = S_2^{\prime}$, then it is a \textit{$S_1$-algebra automorphism}.

The set $\boldsymbol{T} := \boldsymbol{R} \cup \{ -\infty \}$ with two tropical operations:
\[
a \oplus b := \operatorname{max}\{ a, b \} \quad	\text{and} \quad a \odot b := a + b,
\]
where both $a$ and $b$ are in $\boldsymbol{T}$, becomes a semifield.
Here, for any $a \in \boldsymbol{T}$, we handle $-\infty$ as follows:
\[
a \oplus (-\infty) = (-\infty) \oplus a = a \quad \text{and} \quad a \odot (-\infty) = (-\infty) \odot a = -\infty.
\]
$\boldsymbol{T}$ is called the \textit{tropical semifield}.

\subsection{Tropical curves}
	\label{subsection2.2}

In this paper, a \textit{graph} is an unweighted, undirected, finite, connected nonempty multigraph that may have loops.
For a graph $G$, the set of vertices is denoted by $V(G)$ and the set of edges by $E(G)$.
The \textit{valence} of a vertex $v$ of $G$ is the number of edges incident to $v$, where each loop is counted twice.
A vertex $v$ of $G$ is a \textit{leaf end} if $v$ has valence one.
A \textit{leaf edge} is an edge of $G$ incident to a leaf end.

An \textit{edge-weighted graph} $(G, l)$ is the pair of a graph $G$ and a function $l: E(G) \to {\boldsymbol{R}}_{>0} \cup \{\infty\}$, where $l$ can take the value $\infty$ on only leaf edges.
A \textit{tropical curve} is the underlying topological space of an edge-weighted graph $(G, l)$ together with an identification of each edge $e$ of $G$ with the closed interval $[0, l(e)]$.
The interval $[0, \infty]$ is the one point compactification of the interval $[0, \infty)$.
We regard $[0, \infty]$ not just as a topological space but as almost a metric space.
The distance between $\infty$ and any other point is infinite.
When $l(e)=\infty$, the leaf end of $e$ must be identified with $\infty$.
If $E(G) = \{ e \}$ and $l(e)=\infty$, then we can identify either leaf ends of $e$ with $\infty$.
When a tropical curve $\Gamma$ is obtained from $(G, l)$, the edge-weighted graph $(G, l)$ is called a \textit{model} for $\Gamma$.
There are many possible models for $\Gamma$.
A model $(G, l)$ is \textit{loopless} if $G$ is loopless.
For a point $x$ of a tropical curve $\Gamma$, if $x$ is identified with $\infty$, then $x$ is called a \textit{point at infinity}, else, $x$ is called a \textit{finite point}.
$\Gamma_{\infty}$ denotes the set of all points at infinity of $\Gamma$.
If $x$ is a finite point, then the \textit{valence} $\operatorname{val}(x)$ is the number of connected components of $U \setminus \{ x \}$ with any sufficiently small connected neighborhood $U$ of $x$, if $x$ is a point at infinity, then $\operatorname{val}(x) := 1$.
Remark that this ``valence" is defined for a point of a tropical curve and the ``valence" in the first paragraph of this subsection is defined for a vertex of a graph, and these are compatible with each other.
We construct a model $(G_{\circ}, l_{\circ})$ called the {\it canonical model} for $\Gamma$ as follows.
Generally, we determine $V(G_{\circ}) := \{ x \in \Gamma \,|\, \operatorname{val}(x) \not= 2 \}$ except the following two cases.
When $\Gamma$ is homeomorphic to a circle $S^1$, we determine $V(G_{\circ})$ as the set consisting of one arbitrary point of $\Gamma$.
When $\Gamma$ has the edge-weighted graph $(T, l)$ as its model, where $T$ is a tree consisting of three vertices and two edges and $l(E(T)) = \{ \infty \}$, we determine $V(G_{\circ})$ as the set of two points at infinity and any finite point of $\Gamma$.
The \textit{genus} $g(\Gamma)$ of $\Gamma$ is the first Betti number of $\Gamma$, which coincides with $\# E(G) - \# V(G) + 1$ for \textit{any} model $(G, l)$ for $\Gamma$.
For a point $x$ of $\Gamma$, a \textit{half-edge} of $x$ is a connected component of $U \setminus \{ x \}$ with any connected neighborhood $U$ of $x$ which consists of only two valent points and $x$.
We frequently identify a vertex (resp. an edge) of $G$ with the corresponding point (resp. the corresponding closed subset) of $\Gamma$.
The word ``an edge of $\Gamma$" means an edge of $G_{\circ}$.

\subsection{Rational functions and chip-firing moves}
	\label{subsection2.3}

Let $\Gamma$ be a tropical curve.
A continuous map $f : \Gamma \to \boldsymbol{R} \cup \{ \pm \infty \}$ is a \textit{rational function} on $\Gamma$ if $f$ is a constant function of $-\infty$ or a piecewise affine function with integer slopes, with a finite number of pieces and that can take the value $\pm \infty$ at only points at infinity.
For a point $x$ of $\Gamma$ and a rational function $f \in \operatorname{Rat}(\Gamma) \setminus \{ -\infty \}$, $x$ is a \textit{zero} (resp. \textit{pole}) of $f$ if the sign of the sum of outgoing slopes of $f$ at $x$ is plus (resp. minus).
The absolute value of the sum is its \textit{degree}.
If $x$ is a point at infinity, then we regard the outgoing slope of $f$ at $x$ as the slope of $f$ from $y$ to $x$ times minus one, where $y$ is a finite point on the leaf edge incident to $x$ such that $f$ has a constant slope on the interval $(y, x)$.
$\operatorname{Rat}(\Gamma)$ denotes the set of all rational functions on $\Gamma$.
For rational functions $f, g \in \operatorname{Rat}(\Gamma)$ and a point $x \in \Gamma \setminus \Gamma_{\infty}$, we define
\[
(f \oplus g) (x) := \operatorname{max}\{f(x), g(x)\} \quad \text{and} \quad (f \odot g) (x) := f(x) + g(x).
\]
We extend $f \oplus g$ and $f \odot g$ to points at infinity to be continuous on whole $\Gamma$.
Then both are rational functions on $\Gamma$.
Note that for any $f \in \operatorname{Rat}(\Gamma)$, we have
\[
f \oplus (-\infty) = (-\infty) \oplus f = f
\]
and
\[
f \odot (-\infty) = (-\infty) \odot f = -\infty.
\]
Then $\operatorname{Rat}(\Gamma)$ becomes a semifield with these two operations.
Also, $\operatorname{Rat}(\Gamma)$ becomes a $\boldsymbol{T}$-algebra with the natural inclusion $\boldsymbol{T} \hookrightarrow \operatorname{Rat}(\Gamma)$.
Let $\operatorname{Aut}_{\rm semiring}(\operatorname{Rat}(\Gamma))$ (resp. $\operatorname{Aut}_{\boldsymbol{T}}(\operatorname{Rat}(\Gamma))$) denote the set of all semiring (resp. $\boldsymbol{T}$-algebra) automorphisms of $\operatorname{Rat}(\Gamma)$.
Then both $\operatorname{Aut}_{\rm semiring}(\operatorname{Rat}(\Gamma))$ and $\operatorname{Aut}_{\boldsymbol{T}}(\operatorname{Rat}(\Gamma))$ have a group structure.
Note that for $f, g \in \operatorname{Rat}(\Gamma)$, $f = g$ means that $f(x) = g(x)$ for any $x \in \Gamma$.

A \textit{subgraph} of a tropical curve is a compact subset of the tropical curve with a finite number of connected components.
Let $\Gamma_1$ be a subgraph of a tropical curve $\Gamma$ which has no connected components consisting of only a point at infinity, and $l$ a positive number or infinity.
The \textit{chip-firing move} by $\Gamma_1$ and $l$ is defined as the rational function $\operatorname{CF}(\Gamma_1; l)(x) := - \operatorname{min}(l, \operatorname{dist}(x, \Gamma_1))$ with $x \in \Gamma$.

\subsection{Finite harmonic morphisms}
	\label{subsection2.4}

Let $\varphi : \Gamma \to \Gamma^{\prime}$ be a continuous map between tropical curves.
$\varphi$ is a \textit{finite harmonic morphism} if there exist loopless models $(G, l)$ and $(G^{\prime}, l^{\prime})$ for $\Gamma$ and $\Gamma^{\prime}$, respectively, such that $(1)$ $\varphi(V(G)) \subset V(G^{\prime})$ holds, $(2)$ $\varphi(E(G)) \subset E(G^{\prime})$ holds, $(3)$ for any edge $e$ of $G$, there exists a positive integer $\operatorname{deg}_e(\varphi)$ such that for any points $x, y$ of $e$, $\operatorname{dist}(\varphi (x), \varphi (y)) = \operatorname{deg}_e(\varphi) \cdot \operatorname{dist}(x, y)$ holds, and $(4)$ for every vertex $v$ of $G$, the sum $\sum_{e \in E(G):\, e \mapsto e^{\prime},\, v \in e} \operatorname{deg}_e(\varphi)$ is independent of the choice of $e^{\prime} \in E(G^{\prime})$ incident to $\varphi(v)$.
This sum is denoted by $\operatorname{deg}_v(\varphi)$.
Then, the sum $\sum_{v \in V(G):\, v \mapsto v^{\prime}} \operatorname{deg}_v(\varphi)$ is independent of the choice of a vertex $v^{\prime}$ of $G^{\prime}$, and is called the \textit{degree} of $\varphi$, written by $\operatorname{deg}(\varphi)$.
If both $\Gamma$ and $\Gamma^{\prime}$ are singletons, we regard $\varphi$ as a finite harmonic morphism that can have any positive integer as its degree.
Note that if $\varphi \circ \psi$ is a composition of finite harmonic morphisms, then it is also a finite harmonic morphism of degree $\operatorname{deg}(\varphi) \cdot \operatorname{deg}(\psi)$, and thus tropical curves and finite harmonic morphisms between them make a category.
$\operatorname{Aut}(\Gamma)$ denotes the set of all \textit{automorphisms} of $\Gamma$, i.e., finite harmonic morphisms of degree one $\Gamma \to \Gamma$.

\begin{rem}
	\label{rem1}
\upshape{
$\operatorname{Aut}(\Gamma)$ coincides with the set of all continuous maps $\Gamma \to \Gamma$ whose restrictions on $\Gamma \setminus \Gamma_{\infty}$ are isometries $\Gamma \setminus \Gamma_{\infty} \to \Gamma \setminus \Gamma_{\infty}$. 
}
\end{rem}

\section{Main results}
	\label{section3}

In this section, we will prove all assertions in Section \ref{section1}.

We start our consideration from semiring automorphisms of $\boldsymbol{T}$.

\begin{prop}
	\label{prop0}
For $r \in \boldsymbol{R}_{>0}$, the $r$-expansive map $\boldsymbol{T} \to \boldsymbol{T}; t \mapsto r \cdot t$ is a semiring automorphism.
\end{prop}

\begin{proof}
The proof is straightforward.
\end{proof}

Does the converse of Proposition \ref{prop0} hold?
The answer is yes.

\begin{lemma}
	\label{lem2}
Let $S_1, S_2$ be additively idempotent semirings.
A semiring homomorphism $f : S_1 \to S_2$ is order preserving.
\end{lemma}

\begin{proof}
The proof is straightforward.
\end{proof}

\begin{prop}
	\label{prop1}
For a semiring automorphism $\psi : \boldsymbol{T} \to \boldsymbol{T}$, $\psi(1) \in \boldsymbol{R}_{>0}$ and $\psi$ is the $\psi(1)$-expansive map $\boldsymbol{T} \to \boldsymbol{T}; t \mapsto \psi(1) \cdot t$.
\end{prop}

\begin{proof}
Since $\psi$ is a semiring homomorphim, for any $n \in \boldsymbol{Z} \setminus \{ 0 \}$ and $t \in \boldsymbol{R}$, we have $\psi(n \cdot t) = \psi(t^{\odot n}) = \psi(t)^{\odot n} = n \cdot \psi(t)$.

Assume that for any $r \in \boldsymbol{R}_{>0}$, $\psi$ is not the $r$-expansive map.
Then, there exist $t < t^{\prime} \in \boldsymbol{R} \setminus \{ 0 \}$ such that $\frac{t^{\prime}}{t} \not= \frac{\psi(t^{\prime})}{\psi(t)}$.
Note that since $t \not= 0$, so is $\psi(t)$.

Suppose that $t < t^{\prime} < 0$.
By Lemma \ref{lem2}, we have $\psi(t) < \psi(t^{\prime}) < 0$.
Thus the signs of ratios $\frac{t^{\prime}}{t}$ and $\frac{\psi(t^{\prime})}{\psi(t)}$ coincide.
Since $\frac{t^{\prime}}{t} = \frac{-t^{\prime}}{-t}$ and $\frac{\psi(t^{\prime})}{\psi(t)} = \frac{-\psi(t^{\prime})}{-\psi(t)} = \frac{\psi(t^{\prime})^{\odot(-1)}}{\psi(t)^{\odot(-1)}} = \frac{\psi(-t^{\prime})}{\psi(-t)}$, we have $\frac{-t^{\prime}}{-t} \not= \frac{\psi(-t^{\prime})}{\psi(-t)}$.
Hence by replacing $t, t^{\prime}$ with $-t, -t^{\prime}$, respectively, we can assume that $0 < t < t^{\prime}$ in this case.

Suppose that $t < 0 < t^{\prime}$.
If $\frac{t^{\prime}}{t} = -1$, then $t^{\prime} = -t$.
Therefore, $\frac{\psi(t^{\prime})}{\psi(t)} = \frac{\psi(-t)}{\psi(t)} = \frac{-\psi(t)}{\psi(t)} = -1$, and this is a contradiction.
Thus $t^{\prime} \not= -t$, and $0 < -t < t^{\prime}$ or $0 < t^{\prime} < -t$.
Hence $\frac{t^{\prime}}{-t} = - \frac{t^{\prime}}{t} \not= - \frac{\psi(t^{\prime})}{\psi(t)} = \frac{\psi(t^{\prime})}{-\psi(t)} = \frac{\psi(t^{\prime})}{\psi(-t)}$.
From these, it is enough to consider when $0 < t < t^{\prime}$.

Assume that $0 < t < t^{\prime}$.
Let $a:= \frac{t^{\prime}}{t}$.
Then, $a > 1$.
Since $0 < \psi(t) < \psi(t^{\prime})$ by Lemma \ref{lem2}, if $b$ denotes the ratio $\frac{\psi(t^{\prime})}{\psi(t)}$, then $b > 1$.
By assumption, $a \not= b$.

Assume that $a > b$.
By the continuity of real numbers, there exists a rational number $c \in \boldsymbol{Q}_{>1}$ such that $1 < b < c < a$.
By the Archimedean property of the real numbers, there exists a positive integer $m \in \boldsymbol{Z}_{\ge 1}$ such that $m \cdot c > a$.
We can assume that $m \cdot c \in \boldsymbol{Z}_{\ge 1}$.
Since $t > 0$, we have 
\[
0 < t < b \cdot t < c \cdot t < a \cdot t = t^{\prime} < m \cdot c \cdot t.
\]
By Lemma \ref{lem2}, we have
\begin{eqnarray*}
0 &<& \psi(t) < \psi(b \cdot t) < \psi(c \cdot t) < \psi(a \cdot t) = \psi(t^{\prime}) = b \cdot \psi(t)\\
 &<& \psi(m \cdot c \cdot t) = m \cdot \psi(c \cdot t) = m \cdot c \cdot \psi(t).
\end{eqnarray*}
From these, we have $\psi(c \cdot t) = c \cdot \psi(t)$ and $c \cdot \psi(t) < b \cdot \psi(t)$.
Since $0 < \psi(t)$, this means that $c < b$, a contradiction.
Hence $a \ge b$.
By a similar argument, we have $a \le b$.
Therefore, $a = b$ holds but this contradicts $a \not= b$.
Thus there are no such $t < t^{\prime}$.
This means that for any $t < t^{\prime} \in \boldsymbol{R} \setminus \{ 0 \}$, $\frac{t^{\prime}}{t} = \frac{\psi(t^{\prime})}{\psi(t)}$, and thus $\frac{\psi(t)}{t} = \frac{\psi(t^{\prime})}{t^{\prime}}$.
Hence $\frac{\psi(t)}{t} = \frac{\psi(1)}{1} > 0$.
In conclusion, $\psi$ is the $\psi(1)$-expansive map.
\end{proof}

\begin{cor}
	\label{cor4}
The semiring automorphism group of $\boldsymbol{T}$ is isomorphic to the group consisting of the $r$-expansive maps of $\boldsymbol{T}$ for all $r \in \boldsymbol{R}_{>0}$.
\end{cor}

In what follows, we will consider semiring isomorphisms between rational function semifields of tropical curves.

\begin{lemma}
	\label{lem3}
Let $\Gamma_1, \Gamma_2$ be tropical curves.
If $\psi : \operatorname{Rat}(\Gamma_1) \to \operatorname{Rat}(\Gamma_2)$ is a semiring isomorphism, then $\psi(\boldsymbol{T}) = \boldsymbol{T}$.
\end{lemma}

\begin{proof}
Let $t \in \boldsymbol{T}$ and $n \in \boldsymbol{Z}_{>0}$.
Then $\left( \frac{t}{n} \right)^{\odot n} = t$.
Thus we have
\[
\psi(t) = \psi \left( \left(\frac{t}{n} \right)^{\odot n} \right) = \psi \left(\frac{t}{n} \right)^{\odot n} = n \cdot \psi \left(\frac{t}{n} \right).
\]
Therefore each slope of $\psi \left(\frac{t}{n} \right)$ is $\frac{1}{n}$ times that of $\psi(t)$.
Since each rational function on a tropical curve has only integer slopes, $\psi(t)$ must be in $\boldsymbol{T}$.
Thus $\psi(\boldsymbol{T}) \subset \boldsymbol{T}$.
Similarly, we have $\psi^{-1}(\boldsymbol{T}) \subset \boldsymbol{T}$.
Hence $\psi(\boldsymbol{T}) = \boldsymbol{T}$.
\end{proof}

\begin{rem}
	\label{rem2}
\upshape{
By Proposition \ref{prop1} and Lemma \ref{lem3}, the restriction $\psi|_{\boldsymbol{T}}$ is the $\psi(1)$-expansive map $\boldsymbol{T} \to \boldsymbol{T}$, where $\psi$ is as in Lemma \ref{lem3}.
In particular, if $\psi$ is a $\boldsymbol{T}$-algebra isomorphism, then $\psi(1) = 1$.
}
\end{rem}

Let $\Gamma_1, \Gamma_2$ be tropical curves and $\psi : \operatorname{Rat}(\Gamma_1) \to \operatorname{Rat}(\Gamma_2)$ a semiring isomorphism.
Let $r := \psi(1)$.
The following lemma is easy but fundamental.

\begin{lemma}
	\label{lem4}
For any $f \in \operatorname{Rat}(\Gamma_1)$, the following hold:

$(1)$ $r \cdot \operatorname{max}\{ f(x) \,|\, x \in \Gamma_1 \} = \operatorname{max}\{ \psi(f)(x^{\prime}) \,|\, x^{\prime} \in \Gamma_2 \}$, and

$(2)$ $r \cdot \operatorname{min}\{ f(x) \,|\, x \in \Gamma_1 \} = \operatorname{min}\{ \psi(f)(x^{\prime}) \,|\, x^{\prime} \in \Gamma_2 \}$.
\end{lemma}

\begin{proof}
If $f \in \boldsymbol{T}$, the assertions are clear.

Assume that $f$ is not a constant function.
Let $a$ be the maximum value of $f$.
In this case, $a$ is in $\boldsymbol{R} \cup \{ \infty \}$.

Assume $a \in \boldsymbol{R}$.
We have
\[
f \oplus b
\begin{cases}
= b \quad \text{if } b \ge a,\\
\not= b \quad \text{if } b < a.
\end{cases}
\]
Therefore we have
\begin{eqnarray*}
\psi(f) \oplus r \cdot b = \psi(f) \oplus \psi(b) = \psi(f \oplus b)\\
\begin{cases}
= \psi(b) = r \cdot b \	\	\text{if } b \ge a,\\
\not= \psi(b) = r \cdot b \	\	\text{if } b < a.
\end{cases}
\end{eqnarray*}
Thus the maximum value of $\psi(f)$ is $r \cdot a$.

Assume $a = \infty$.
Then for any $t \in \boldsymbol{T}$, we have $f \oplus t \not= t$.
Thus 
\[
\psi(f) \oplus r \cdot t = \psi(f) \oplus \psi(t) = \psi(f \oplus t) \not= \psi(t) = r \cdot t
\]
hold.
This means that the maximum value of $\psi(f)$ is $\infty$ since tropical curves are compact.

For the minimum values of $f$ and $\psi(f)$, we can obtain the conclusion by applying the maximum value case for $f^{\odot (-1)} = -f$ and $\psi(f^{\odot (-1)}) = -\psi(f)$ since 
\[
\operatorname{min}\{ f(x) \,|\, x \in \Gamma_1 \} = - \operatorname{max}\{ -f(x) \,|\, x \in \Gamma_1 \}
\]
and 
\[
\operatorname{min}\{ \psi(f)(x^{\prime}) \,|\, x^{\prime} \in \Gamma_2 \} = - \operatorname{max}\{ -\psi(f)(x^{\prime}) \,|\, x^{\prime} \in \Gamma_2 \}. \qedhere
\]
\end{proof}

Lemma \ref{lem5} is our key lemma to prove Theorem \ref{thm1}.

\begin{lemma}
	\label{lem5}
$(1)$ For any $x \in \Gamma_1 \setminus \Gamma_{1 \infty}$, there exist $\varepsilon > 0$ and a unique point $x^{\prime} \in \Gamma_2 \setminus \Gamma_{2 \infty}$ such that $\psi(\operatorname{CF}(\{ x \}; \varepsilon)) = \operatorname{CF}(\{ x^{\prime} \}; \psi(1) \cdot \varepsilon)$.
Furthermore, $\operatorname{val}(x) = \operatorname{val}(x^{\prime})$ holds.

$(2)$ For any $x \in \Gamma_{1 \infty}$, there exists a finite point $y$ on the unique edge $e$ incident to $x$ and there exists a unique point $x^{\prime} \in \Gamma_{2 \infty}$, such that the point $y^{\prime} \in \Gamma_2 \setminus \Gamma_{2 \infty}$ corresponding to $y$ is on the unique edge incident to $x^{\prime}$ and $\psi(\operatorname{CF}(\Gamma_1 \setminus (y, x]; \infty)) = \operatorname{CF}(\Gamma_2 \setminus (y^{\prime}, x^{\prime}]; \infty)$.
\end{lemma}

In the proof of Lemma \ref{lem5}, we will consider several times rational functions of the form $(f^{\odot(-1)} \oplus a)^{\odot(-1)}$ for a rational function $f$ on a tropical curve $\Gamma$ and $a \in \boldsymbol{R}$.
It is equal to the minimum of $f$ and $-a$, i.e., for any $x \in \Gamma$, 
\[
(f^{\odot(-1)} \oplus a)^{\odot(-1)}(x) = \operatorname{min}\{ f(x), -a \},
\]
where we define $\operatorname{min} \{ \infty, t \} = t$ and $\operatorname{min} \{ -\infty, t \} = -\infty$ for any $t \in \boldsymbol{R}$.
Also, if 
\[
f = \bigoplus_{i = 1}^n g_i
\]
and for any $j$,
\[
f \not= \bigoplus_{i \not= j, i = 1}^n g_i
\]
hold, then we call this finite sum $\bigoplus_{i = 1}^n g_i$ an \textit{irredundant representation} of $f$.
The image of an irredundant representation by a semiring isomorphism is also an irredundant representation.
In particular, if $f$ is a chip-firing move of the form of Lemma \ref{lem5}$(1)$ on $\Gamma_1$ and if $\bigoplus_{i = 1}^n g_i$ is an irredundant representation of $f$ satisfying the following condition $(\star)$, then $n \le \operatorname{val}(x)$.
\[
(\star)
\begin{cases}
(1)~\text{The maximum value of each }g_i \text{ is zero.}\\
(2)~\text{The minimum value of each }g_i \text{ is }-\varepsilon.\\
(3)~\text{Each }g_i \text{ has a unique zero on each half-edge of } x \text{ in the }\varepsilon\text{-ne-}\\\text{ighborhood of }x.
\end{cases}
\]
Similarly, if $f$ is a chip-firing move of the form of Lemma \ref{lem5}$(2)$ on $\Gamma_1$ and if $\bigoplus_{i = 1}^n g_i$ is an irredundant representation of $f^{\odot (-1)}$ satisfying the following condition $(\star \star)$, then $n = \operatorname{val}(x) = 1$.
\[
(\star \star)
\begin{cases}
(1)~\text{The maximum value of each }g_i \text{ is }\infty.\\
(2)~\text{The minimum value of each }g_i \text{ is zero.}\\
(3)~\text{Each }g_i \text{ has a unique zero on } [y, x).
\end{cases}
\]

\begin{proof}[Proof of Lemma \ref{lem5}]
We shall prove $(1)$.
For any $\varepsilon > 0$, we consider $\psi(\operatorname{CF}(\{ x \}; \varepsilon))$.
By Lemma \ref{lem4}, the maximum value of $\psi(\operatorname{CF}(\{ x \}; \varepsilon))$ is zero.
Let $\operatorname{Max}^{\prime}$ denote the set of all points of $\Gamma_2$ where $\psi(\operatorname{CF}(\{ x \}; \varepsilon))$ attains the maximum value $0$.
Since $\operatorname{CF}(\{ x \}; \varepsilon)$ is not a constant function, $\psi(\operatorname{CF}(\{ x \}; \varepsilon))$ is also not a constant function.
Therefore the boundary set $\partial \operatorname{Max}^{\prime}$ of $\operatorname{Max}^{\prime}$ is a non-empty finite set.
For $\delta$ such that $0 < \delta < \varepsilon$, since 
\[
\operatorname{CF}(\{ x \}; \delta) = \operatorname{CF}(\{ x \}; \varepsilon) \oplus (-\delta),
\]
we have 
\[
\psi(\operatorname{CF}(\{ x \}; \delta)) = \psi(\operatorname{CF}(\{ x \}; \varepsilon)) \oplus (-r \cdot \delta).
\]
Therefore, if necessary, by replacing $\varepsilon$ with $\delta$, we may assume that $\psi(\operatorname{CF}(\{ x \}; \varepsilon))$ has a unique zero on each half-edge of each point $x^{\prime} \in \partial \operatorname{Max}^{\prime}$ in the $(r \cdot \varepsilon)$-neighborhood of $\operatorname{Max}^{\prime}$.
Also, when $\psi(\operatorname{CF}(\{ x \}; \varepsilon))$ has a pole outside of $\operatorname{Max}^{\prime}$, by replacing $\varepsilon$ with the $\frac{1}{r}$ times the maximum value of $\psi(\operatorname{CF}(\{ x \}; \varepsilon))$ on such poles as above $\delta$, we may assume that there are no such poles.
Hereafter, we assume these two conditions.

Assume that $\# \operatorname{Max}^{\prime}$ is infinite.
Since $\psi(\operatorname{CF}(\{ x \}; \varepsilon))$ is a rational function, $\operatorname{Max}^{\prime}$ has a finite number of connected components.
Thus $\operatorname{Max}^{\prime}$ has a connected component that is an infinite set.
For such a connected component $\Gamma^{\prime}$, let $y^{\prime} \in \Gamma^{\prime} \setminus (\partial \Gamma^{\prime} \cup \Gamma_{2\infty})$.
Let $0 < \delta^{\prime} \le \operatorname{min} \{ \operatorname{dist}(y^{\prime}, z^{\prime}) \,|\, z^{\prime} \in \partial \Gamma^{\prime} \}$.
Let $(\ast)$ denote the rational function
\[
\psi^{-1}(\psi(\operatorname{CF}(\{ x \}; \varepsilon)) \odot \operatorname{CF}(\{ y^{\prime} \}; \delta^{\prime}) \odot \delta^{\prime}).
\]
Since
\[
\left( \left\{ \psi (\operatorname{CF}(\{ x \}; \varepsilon )) \odot \operatorname{CF}(\{ y^{\prime} \}; \delta^{\prime}) \odot \delta^{\prime} \right\}^{\odot (-1)} \oplus 0 \right)^{\odot (-1)} = \psi(\operatorname{CF}(\{ x \}; \varepsilon)),
\]
we have
\begin{eqnarray*}
&&\operatorname{CF}(\{ x \}; \varepsilon)\\
&=& \psi^{-1} \left( \left[ \left\{ \psi(\operatorname{CF}(\{ x \}; \varepsilon)) \odot \operatorname{CF}(\{ y^{\prime} \}; \delta^{\prime}) \odot \delta^{\prime} \right\}^{\odot (-1)} \oplus 0 \right]^{\odot (-1)} \right)\\
&=& \left(\psi^{-1} \left( \left\{ \psi(\operatorname{CF}(\{ x \}; \varepsilon)) \odot \operatorname{CF}(\{ y^{\prime} \}; \delta^{\prime}) \odot \delta^{\prime} \right\}^{\odot (-1)} \oplus 0 \right) \right)^{\odot (-1)}\\
&=& \left( (\ast)^{\odot (-1)} \oplus 0 \right)^{\odot (-1)}.
\end{eqnarray*}
Therefore $(\ast)$ must be $\operatorname{CF}(\{ x \}; \varepsilon)$.
However this contradicts that the maximum value of $(\ast)$ is $\frac{\delta^{\prime}}{r}$ by Lemma \ref{lem4}.
Hence $\operatorname{Max}^{\prime}$ is a finite set.

If $\operatorname{Max}^{\prime} \cap \Gamma_{2 \infty} \not= \varnothing$, then for any $x^{\prime} \in \operatorname{Max}^{\prime} \cap \Gamma_{2 \infty}$, $\psi(\operatorname{CF}(\{ x \}; \varepsilon))$ takes the maximum value $0$ at $x^{\prime}$ and has $x^{\prime}$ as a pole.
However, by the definition of a rational function, this does not occur.
Thus $\operatorname{Max}^{\prime} \cap \Gamma_{2 \infty}$ is empty.

For any $x^{\prime} \in \operatorname{Max}^{\prime}$, let $g^{\prime}_{x^{\prime}}$ be the rational function on $\Gamma_2$ that coincides with $\psi(\operatorname{CF}(\{ x \}; \varepsilon))$ in the $(r \cdot \varepsilon)$-neighborhood of $x^{\prime}$ and is $-r \cdot \varepsilon$ on other points.
Then we have 
\[
\psi(\operatorname{CF}(\{ x \}; \varepsilon)) = \bigoplus_{x^{\prime} \in \operatorname{Max}^{\prime}} g^{\prime}_{x^{\prime}}.
\]
The right-hand side is an irredundant representation.
Hence we have
\[
\operatorname{CF}(\{ x \}; \varepsilon) = \bigoplus_{x^{\prime} \in \operatorname{Max}^{\prime}} \psi^{-1}(g^{\prime}_{x^{\prime}}).
\]
This right-hand side is also an irredundant representation.
If necessary, by replacing $\varepsilon$ with a smaller positive number, we may assume that this irredundant representation satisfies the condition $(\star)$.
Thus the number of elements of $\operatorname{Max}^{\prime}$ is at most the valence of $x$.
In particular, $\operatorname{val}(x) = 1$ implies $\# \operatorname{Max}^{\prime} = 1$.

We will show that if $\operatorname{val}(x) = 1$, then $\operatorname{val}(x^{\prime}) = 1$ for the unique element $x^{\prime}$ of $\operatorname{Max}^{\prime}$ and $\psi(\operatorname{CF}(\{ x \}; \varepsilon)) = \operatorname{CF}(\{ x^{\prime} \}; r \cdot \varepsilon)$.
Assume that $\operatorname{val}(x^{\prime}) \ge 2$.
Let $m = \operatorname{val}(x^{\prime})$ and $h_1^{\prime}, \ldots, h_m^{\prime}$ be the  half-edges of $x^{\prime}$ in the $(r \cdot \varepsilon)$-neighborhood of $x^{\prime}$.
For each $i = 1, \ldots, m$, let $g^{\prime}_i$ be the rational function on $\Gamma_2$ defined as follows:
$g^{\prime}_i$ coincides with $\psi(\operatorname{CF}(\{ x \}; \varepsilon))$ on $h_i^{\prime}$, and, for any $j \not= i$, $g^{\prime}_i$ has the slope $s (< 0)$ of $\psi(\operatorname{CF}(\{ x \}; \varepsilon))$ minus one on the intersection of $h^{\prime}_j$ and the $(r \cdot \varepsilon / (|s| + 1) )$-neighborhood of $x^{\prime}$, and, $g^{\prime}_i$ takes the minimum value $- r \cdot \varepsilon$ on other points.
Then we have 
\[
\psi(\operatorname{CF}(\{ x \}; \varepsilon)) = \bigoplus_{i = 1}^{m} g^{\prime}_i.
\]
The right-hand side is an irredundant representation.
Hence we obtain the irredundant representation $\bigoplus_{i = 1}^{m} \psi^{-1}(g^{\prime}_i)$ of $\operatorname{CF}(\{ x \}; \varepsilon)$ consisting of $m$ elements.
Again, if necessary, by replacing $\varepsilon$ with a smaller positive number, we may assume that this irredundant representation satisfies the condition $(\star)$.
This means that $m \le 1$, a contradiction.
Thus $\operatorname{val}(x^{\prime})$ must be one.
Also if $\psi(\operatorname{CF}(\{ x \}; \varepsilon))$ is not $\operatorname{CF}(\{ x^{\prime} \}; r \cdot \varepsilon)$, then 
\[
\psi(\operatorname{CF}(\{ x \}; \varepsilon)) \oplus \operatorname{CF}(\{ x^{\prime} \}; r \cdot \varepsilon) = \operatorname{CF}(\{ x^{\prime} \}; r \cdot \varepsilon).
\]
Then, we have 
\[
\operatorname{CF}(\{ x \}; \varepsilon) \not= \psi^{-1}(\operatorname{CF}(\{ x^{\prime} \}; r \cdot \varepsilon))
\]
and
\[
\operatorname{CF}(\{ x \}; \varepsilon) \oplus \psi^{-1}(\operatorname{CF}(\{ x^{\prime} \}; r \cdot \varepsilon)) = \psi^{-1}(\operatorname{CF}(\{ x^{\prime} \}; r \cdot \varepsilon)).
\]
However, since $\psi^{-1}(\operatorname{CF}(\{ x^{\prime} \}; r \cdot \varepsilon))$ attains the maximum value $0$ at $x$ and only at $x$, and since the minimum value of $\psi^{-1}(\operatorname{CF}(\{ x^{\prime} \}; r \cdot \varepsilon))$ is $-\varepsilon$ by Lemma \ref{lem4}, it does not occur.
Therefore $\psi(\operatorname{CF}(\{ x \}; \varepsilon))$ must be $\operatorname{CF}(\{ x^{\prime} \}; r \cdot \varepsilon)$.

Assume that when $\operatorname{val}(x) = k$, the assertion holds.
Consider the case $\operatorname{val}(x) = k + 1$.

Suppose that $\operatorname{Max}^{\prime}$ contains a point $y^{\prime}$ of valence at most $k$.
If necessary, by replacing $\varepsilon$ with a sufficiently small positive number, by assumption, there exists a unique point $y \in \Gamma_1 \setminus \Gamma_{1 \infty}$ such that 
\[
\psi^{-1}(\operatorname{CF}(\{ y^{\prime} \}; r \cdot \varepsilon)) = \operatorname{CF}(\{ y \}; \varepsilon)
\]
and $k \ge \operatorname{val}(y^{\prime}) = \operatorname{val}(y)$.
Thus $x \not= y$.
On the other hand, since 
\[
g^{\prime}_{y^{\prime}} \oplus \operatorname{CF}(\{ y^{\prime} \}; r \cdot \varepsilon) = \operatorname{CF}(\{ y^{\prime} \}; r \cdot \varepsilon)
\]
and
\[
\bigoplus_{x^{\prime} \in \operatorname{Max}^{\prime}} g^{\prime}_{x^{\prime}} = \psi(\operatorname{CF}(\{ x \}; \varepsilon)),
\]
we have
\[
\psi^{-1} (g^{\prime}_{y^{\prime}}) \oplus \operatorname{CF}(\{ y \}; \varepsilon) = \operatorname{CF}(\{ y \}; \varepsilon)
\]
and
\[
\bigoplus_{x^{\prime} \in \operatorname{Max}^{\prime}} \psi^{-1}(g^{\prime}_{x^{\prime}}) = \operatorname{CF}(\{ x \}; \varepsilon).
\]
By the latter equality, $\psi^{-1}(g^{\prime}_{y^{\prime}})$ must attain the maximum value $0$ at $x$ and only at $x$.
However, by the former equality, this means $x = y$, a contradiction.
Therefore $\operatorname{Max}^{\prime}$ has no points of valence at most $k$.

Assume that $\# \operatorname{Max}^{\prime} \ge 2$.
For any $x^{\prime} \in \operatorname{Max}^{\prime}$, let $g^{\prime}_{x^{\prime}, 1}, \ldots, g^{\prime}_{x^{\prime}, \operatorname{val}(x^{\prime})}$ be the rational functions on $\Gamma_2$ defined as the above $g^{\prime}_i$s which is not constant only on the $(r \cdot \varepsilon)$-neighborhood of $x^{\prime}$.
Then we have 
\[
\psi(\operatorname{CF}(\{ x \}; \varepsilon)) = \bigoplus_{x^{\prime} \in \operatorname{Max}^{\prime}} \bigoplus_{i = 1}^{\operatorname{val}(x^{\prime})} g^{\prime}_{x^{\prime}, i}.
\]
The right-hand side is an irredundant representation.
This means that if $\varepsilon$ is sufficiently small, then $\operatorname{CF}(\{ x \}; \varepsilon)$ has an irredundant representation 
\[
\bigoplus_{x^{\prime} \in \operatorname{Max}^{\prime}} \bigoplus_{i = 1}^{\operatorname{val}(x^{\prime})} \psi^{-1}(g^{\prime}_{x^{\prime}, i})
\]
consisting of $\sum_{x^{\prime} \in \operatorname{Max}^{\prime}} \operatorname{val}(x^{\prime}) > \operatorname{val}(x)$ elements and satisfying the condition $(\star)$.
This is a contradiction.
Thus $\# \operatorname{Max}^{\prime}$ must be one.
Let $x^{\prime}$ be the unique point in $\operatorname{Max}^{\prime}$.
If $\operatorname{val}(x) < \operatorname{val}(x^{\prime})$, then by a similar arugument above, we have a contradiction.
Therefore $\operatorname{val}(x) = \operatorname{val}(x^{\prime})$.
Also by a similar argument in the case of $\operatorname{val}(x) = 1$, $\psi(\operatorname{CF}\{ x \}; \varepsilon))$ is equal to $\operatorname{CF}(\{ x^{\prime} \}; r \cdot \varepsilon)$.
Hence we have the conclusion.

We shall prove $(2)$.
Let $y$ be a point of $e \setminus \Gamma_{2 \infty}$.
By Lemma \ref{lem4}, the maximum (resp. minimum) value of $\psi(\operatorname{CF}(\Gamma_1 \setminus (y, x]; \infty))$ is $0$ (resp. $-\infty$).
Thus if $\operatorname{Min}^{\prime}$ denotes the set of all points of $\Gamma_2$ where $\psi(\operatorname{CF}(\Gamma_1 \setminus (y, x]; \infty))$ attains the minimum value $-\infty$, then $\operatorname{Min}^{\prime} \subset \Gamma_{2 \infty}$.
Hence $\operatorname{Min}^{\prime}$ is finite.
For a sufficiently large number $a > 0$, there exists a unique point $z \in (y, x)$ such that 
\[
\left\{  (- a) \odot \operatorname{CF}(\Gamma_1 \setminus (y, x]; \infty)^{\odot (-1)} \oplus 0 \right\}^{\odot (-1)} = \operatorname{CF}(\Gamma_1 \setminus (z, x]; \infty).
\]
Thus we have 
\[
\{ (- r \cdot a) \odot \psi(\operatorname{CF}(\Gamma_1 \setminus (y, x]; \infty))^{\odot (-1)} \oplus 0 \}^{\odot (-1)} = \psi(\operatorname{CF}(\Gamma_1 \setminus (z, x]; \infty)).
\]
Hence, if necessary, by replacing $y$ with $z$, we may assume the following:
$\psi(\operatorname{CF}(\Gamma_1 \setminus (y, x]; \infty))$ has no zeros on $\Gamma_2 \setminus \Gamma_{2 \infty}$, has a unique pole on the edge incident to each element of $\operatorname{Min}^{\prime}$, and has no other poles.
Hereafter, we assume these three conditions.
If $\operatorname{P}^{\prime}$ denotes the set of all poles of $\psi(\operatorname{CF}(\Gamma_1 \setminus (y, x]; \infty))$, then there exists a one-to-one correspondence between $\operatorname{P}^{\prime}$ and $\operatorname{Min}^{\prime}$.
Let $y^{\prime} \in \operatorname{P}^{\prime}$ correspond to $x^{\prime} \in \operatorname{Min}^{\prime}$.
Then there exists a unique positive integer $n_{y^{\prime}}$ such that
\begin{eqnarray*}
&&\operatorname{CF}(\Gamma_2 \setminus (y^{\prime}, x^{\prime}]; \infty))^{\odot n_{y^{\prime}}}\\
&=& \psi(\operatorname{CF}(\Gamma_1 \setminus (y, x]; \infty)) \oplus \operatorname{CF}(\Gamma_2 \setminus (y^{\prime}, x^{\prime}]; \infty)^{\odot n_{y^{\prime}}}
\end{eqnarray*}
and
\begin{eqnarray*}
&& \operatorname{CF}(\Gamma_2 \setminus (y^{\prime}, x^{\prime}]; \infty)^{\odot (n_{y^{\prime}} + 1)}\\
&\not=&\psi(\operatorname{CF}(\Gamma_1 \setminus (y, x]; \infty)) \oplus \operatorname{CF}(\Gamma_2 \setminus (y^{\prime}, x^{\prime}]; \infty)^{\odot (n_{y^{\prime}} + 1)}. 
\end{eqnarray*}
By considering the images of these equalities by $\psi^{-1}$, we have $n_{y^{\prime}} = 1$.
Let $g^{\prime}_{y^{\prime}}$ be $\operatorname{CF}(\Gamma_2 \setminus (y^{\prime}, x^{\prime}]; \infty)^{\odot (-1)}$.
Then $\bigoplus_{y^{\prime} \in P^{\prime}} g^{\prime}_{y^{\prime}}$ is an irredundant representation of $\psi(\operatorname{CF}(\Gamma_1 \setminus (y, x]; \infty))^{\odot (-1)}$.
If necessary, by replacing $y$ with $z \in (y, x)$ such that $\operatorname{dist}(y, z)$ is sufficiently large, we may assume that the irredundant repredentation $\bigoplus_{y^{\prime} \in P^{\prime}} \psi^{-1}(g^{\prime}_{y^{\prime}})$ of $\operatorname{CF}(\Gamma_1 \setminus (y, x]; \infty)^{\odot (-1)}$ satisfies the condition $(\star \star)$.
Hence we have $\# P^{\prime} = \# \operatorname{Min}^{\prime} = 1$.
Therefore, we have
\[
\psi({\rm CF}(\Gamma_1 \setminus (y, x]; \infty)) = \operatorname{CF}(\Gamma_2 \setminus (y^{\prime}, x^{\prime}]; \infty).
\]

We will show that $y$ corresponds to the unique pole $y^{\prime} \in P^{\prime}$.
By $(1)$, for a sufficiently small positive number $\varepsilon$, there exists a unique $z^{\prime} \in \Gamma_2 \setminus \Gamma_{2 \infty}$ such that 
\[
\psi(\operatorname{CF}(\{ y \}; \varepsilon)) = \operatorname{CF}(\{ z^{\prime} \}; r \cdot \varepsilon).
\]
Since
\begin{eqnarray*}
&& \operatorname{CF}(\Gamma_2 \setminus (y^{\prime}, x^{\prime}]; \infty) \oplus (-r \cdot \varepsilon)\\
&=& \psi(\operatorname{CF}(\Gamma_1 \setminus (y, x]; \infty)) \oplus (-r \cdot \varepsilon)\\
&=& \psi(\operatorname{CF}(\Gamma_1 \setminus (y, x]; \infty) \oplus (-\varepsilon))\\
&=& \psi(\operatorname{CF}(\Gamma_1 \setminus (y, x]; \infty) \oplus \operatorname{CF}(\{ y \}; \varepsilon))\\
&=& \operatorname{CF}(\Gamma_2 \setminus (y^{\prime}, x^{\prime}]; \infty) \oplus \operatorname{CF}(\{ z^{\prime} \}; r \cdot \varepsilon),
\end{eqnarray*}
we have $z^{\prime} \not\in (y^{\prime}, x^{\prime}]$.
Suppose that $y^{\prime} \not= z^{\prime}$.
If necessary, by replacing $\varepsilon$ with a positive number at most $\frac{1}{r} \cdot \operatorname{dist}(y^{\prime}, z^{\prime})$, we obtain
\begin{eqnarray*}
&&(-r \cdot \varepsilon) \odot \operatorname{CF}(\Gamma_2 \setminus (y^{\prime}, x^{\prime}]; \infty)\\
&=& \left( \left\{ \operatorname{CF}(\Gamma_2 \setminus (y^{\prime}, x^{\prime}]; \infty) \odot \operatorname{CF}( \{ z^{\prime} \}; r \cdot \varepsilon) \right\}^{\odot (-1)} \oplus r \cdot \varepsilon \right)^{\odot (-1)}.
\end{eqnarray*}
On the other hand, the image of this equality by $\psi^{-1}$ is 
\begin{eqnarray*}
&& (-\varepsilon) \odot \operatorname{CF}(\Gamma_1 \setminus (y, x]; \infty) \\
&=& \left( \left\{ \operatorname{CF}(\Gamma_1 \setminus (y, x]; \infty) \odot \operatorname{CF}(\{ y \}; \varepsilon) \right\}^{\odot (-1)} \oplus \varepsilon \right)^{\odot (-1)},
\end{eqnarray*}
and this is a contradiction.
In conclusion, $y^{\prime}$ must be $z^{\prime}$.
\end{proof}

Now, we can prove Theorem \ref{thm1}.

\begin{proof}[Proof of Theorem \ref{thm1}]
Lemma \ref{lem5} induces a map $\varphi : \Gamma_1 \to \Gamma_2; x \mapsto x^{\prime}$.
We will show that this $\varphi$ is a $\psi(1)$-expansive map.
Since $\psi$ is bijective, $\varphi$ is also bijective.
Put $r := \psi(1)$.
For any $x \in \Gamma_1 \setminus \Gamma_{1 \infty}$ and a positive number $\varepsilon$ obtained by Lemma \ref{lem5}(1), let $U$ be the $\varepsilon$-neighborhood of $x$.
For any $y \in U \setminus \{ x \}$, let $d = \operatorname{dist}(x, y)$.
Let $\delta$ be a positive number for $y$ obtained by Lemma \ref{lem5}.
If necessary, we replace $\varepsilon$ with a positive number between $d$ and $d + \delta$.
Then, we have
\[
\operatorname{CF}(\{ x \}; \varepsilon) \oplus \operatorname{CF}(\{ y \}; \delta) \odot a
\begin{cases}
= \operatorname{CF}(\{ x \}; \varepsilon) \quad \text{if } a \le -d,\\
\not= \operatorname{CF}(\{ x \}; \varepsilon) \quad \text{if } a > -d.
\end{cases}
\]
The images of these by $\psi$ are
\begin{eqnarray*}
\operatorname{CF}(\{ \varphi(x) \}; r \cdot \varepsilon) \oplus \operatorname{CF}(\{ \varphi(y) \}; r \cdot \delta) \odot r \cdot a\\
\begin{cases}
= \operatorname{CF}(\{ \varphi(x) \}; r \cdot \varepsilon) \quad \text{if } a \le -d,\\
\not= \operatorname{CF}(\{ \varphi(x) \}; r \cdot \varepsilon) \quad \text{if } a > -d.
\end{cases}
\end{eqnarray*}
Thus $\operatorname{dist}(\varphi(x), \varphi(y)) = r \cdot d$.
Also, for any $z \in \Gamma_1 \setminus \Gamma_{1 \infty}$, let $P$ be a path from $x$ to $z$.
If necessary, by subdividing $P$, we see that $P$ is mapped by $\varphi$ to a path from $\varphi(x)$ to $\varphi(z)$ of length equal to $r$ times the length of $P$.
Since $\varphi$ is injective, distinct paths from $x$ to $z$ are mapped by $\varphi$ to distinct paths from $\varphi(x)$ to $\varphi(z)$.
Since $\varphi$ is surjective, all paths from $\varphi(x)$ to $\varphi(z)$ are images by $\varphi$ of paths from $x$ to $z$.

By Lemma \ref{lem5}(2), $\varphi$ induces a bijection between $\Gamma_{1 \infty}$ and $\Gamma_{2 \infty}$.
Also, by Lemma \ref{lem5}(2) and its proof, when $\Gamma_{1 \infty} \ni x \mapsto x^{\prime} \in \Gamma_{2 \infty}$, $x^{\prime}$ is a (unique) accumulation point of the image $\{ {\varphi(y_i)} \}$ of any sequence $\{ y_i \}$ such that $y_i \to x$ for $i \to \infty$.
This means that $\varphi$ is continuous.
Thus we obtain the conclusion.
\end{proof}

Theorem \ref{thm1} induces several corollaries.

\begin{cor}
	\label{cor5}
Let $\psi$ and $\varphi$ be as in Theorem \ref{thm1}.
Then 
\[
\psi (f) = (r \cdot f) \circ \varphi^{-1}
\]
holds for any $f \in \operatorname{Rat}(\Gamma_1)$.
\end{cor}

In the proof of Corollary \ref{cor5}, we will use the proof of the main theorem of \cite{JuAe3} ``for a tropical curve $\Gamma$, $\operatorname{Rat}(\Gamma)$ is finitely generated over $\boldsymbol{T}$ as a tropical semifield"; in that proof, the author showed that any set $A$ of chip-firing moves on $\Gamma_1$ of the forms in Lemma \ref{lem5} which induces a surjection $A \twoheadrightarrow \Gamma$ generates $\operatorname{Rat}(\Gamma)$ over $\boldsymbol{T}$ as a tropical semifield.

\begin{proof}[Proof of Corollary \ref{cor5}]
Let $f$ and $g$ be chip-firing moves on $\Gamma_1$ of the forms in Lemma \ref{lem5}(1) and (2), respectively.
By Lemma \ref{lem5} and the definition of $\varphi$, we have 
\[
\psi(f) = (r \cdot f) \circ \varphi^{-1} \quad \text{and} \quad \psi(g) = (r \cdot g) \circ \varphi^{-1}.
\]
Since all of such $f$ and $g$ generate $\operatorname{Rat}(\Gamma_1)$ over $\boldsymbol{T}$ as a tropical semifield, for any $h \in \operatorname{Rat}(\Gamma_1)$, we have $\psi(h) = (r \cdot h) \circ \varphi^{-1}$.
In fact, if $\psi(\widetilde{f}) = (r \cdot \widetilde{f}) \circ \varphi^{-1}$ and $\psi(\widetilde{g}) = (r \cdot \widetilde{g}) \circ \varphi^{-1}$, then we have
\begin{eqnarray*}
\psi(\widetilde{f} \oplus \widetilde{g}) &=& \psi(\widetilde{f}) \oplus \psi(\widetilde{g}) \\
&=& (r \cdot \widetilde{f}) \circ \varphi^{-1} \oplus (r \cdot \widetilde{g}) \circ \varphi^{-1} \\
&=& ((r \cdot \widetilde{f}) \oplus (r \cdot \widetilde{g})) \circ \varphi^{-1} \\
&=& (r \cdot (\widetilde{f} \oplus \widetilde{g})) \circ \varphi^{-1}
\end{eqnarray*}
and
\begin{eqnarray*}
\psi(\widetilde{f} \odot \widetilde{g}) &=& \psi(\widetilde{f}) \odot \psi(\widetilde{g}) \\
&=& (r \cdot \widetilde{f}) \circ \varphi^{-1} \odot (r \cdot \widetilde{g}) \circ \varphi^{-1} \\
&=& ((r \cdot \widetilde{f}) \odot (r \cdot \widetilde{g})) \circ \varphi^{-1} \\
&=& (r \cdot (\widetilde{f} \odot \widetilde{g})) \circ \varphi^{-1}.
\end{eqnarray*}
Hence, we have the conclusion.
\end{proof}

It is easy to check that the converse of Theorem \ref{thm1} holds:

\begin{lemma}
	\label{lem6}
Let $\Gamma_1, \Gamma_2$ be tropical curves.
An $r$-expansive map $\varphi : \Gamma_1 \to \Gamma_2$ with $r \in \boldsymbol{R}_{>0}$ induces the semiring isomorphism 
\[
\psi : \operatorname{Rat}(\Gamma_1) \to \operatorname{Rat}(\Gamma_2); \qquad f \mapsto \left( r \cdot f \right) \circ \varphi^{-1}
\]
and $\psi(1) = r$.
\end{lemma}

\begin{proof}
Since $\varphi^{-1}$ is surjective, $\psi(f)$ is a function $\Gamma_2 \to \boldsymbol{R} \cup \{ \pm \infty \}$.
For $x, y \in \Gamma_1$, since $r \cdot \operatorname{dist}(x, y) = \operatorname{dist}(\varphi(x), \varphi(y))$, slopes of $f$ and $\psi(f)$ on intervals corresponding by $\varphi$ coincide.
Thus $\psi(f)$ is a rational function on $\Gamma_2$.

Let 
\[
\theta : \operatorname{Rat}(\Gamma_2) \to \operatorname{Rat}(\Gamma_1); \qquad f^{\prime} \mapsto \left( \frac{1}{r} \cdot f^{\prime} \right) \circ \varphi.
\]
By a similar argument above, we have $\theta (f^{\prime}) \in \operatorname{Rat}(\Gamma_1)$ for any $f^{\prime} \in \operatorname{Rat}(\Gamma_2)$.
By definition, we have
\[
\theta \circ \psi = \operatorname{id}_{\operatorname{Rat}(\Gamma_1)} \quad \text{and} \quad \psi \circ \theta = \operatorname{id}_{\operatorname{Rat}(\Gamma_2)},
\]
where $\operatorname{id}_{\operatorname{Rat}(\Gamma_i)}$ denotes the identity map of $\operatorname{Rat}(\Gamma_i)$.
Hence $\psi$ is bijective.

The proof that $\psi$ is a semiring homomorphism is straightforward.
\end{proof}

Now we can prove Corollary \ref{cor1}:

\begin{proof}[Proof of Corollary \ref{cor1}]
Clearly, both $\mathscr{C}, \mathscr{D}$ are groupoids.

Let 
\[
F: \mathscr{C} \to \mathscr{D}
\]
be
\[
\operatorname{Ob}(\mathscr{C}) \to \operatorname{Ob}(\mathscr{D}); \qquad \Gamma \mapsto \Gamma
\]
and for $\Gamma_1, \Gamma_2 \in \operatorname{Ob}(\mathscr{C})$, 
\[
\operatorname{Hom}_{\mathscr{C}}(\Gamma_1, \Gamma_2) \to \operatorname{Hom}_{\mathscr{D}}(\Gamma_1, \Gamma_2); \qquad \psi \mapsto \varphi,
\]
where $\varphi$ is the $r$-expansive map $\Gamma_1 \to \Gamma_2$ defined as in Theorem \ref{thm1}.
Let 
\[
G: \mathscr{D} \to \mathscr{C}
\]
be
\[
\operatorname{Ob}(\mathscr{D}) \to \operatorname{Ob}(\mathscr{C}); \qquad \Gamma \to \Gamma
\]
 and for $\Gamma_1, \Gamma_2 \in \operatorname{Ob}(\mathscr{D})$, 
\[
\operatorname{Hom}_{\mathscr{D}}(\Gamma_1, \Gamma_2) \to \operatorname{Hom}_{\mathscr{C}}(\Gamma_1, \Gamma_2); \qquad \varphi \mapsto \left( f \mapsto (r \cdot f) \circ \varphi^{-1} \right), 
\]
where $\varphi$ is an $r$-expansive map.
By Theorem \ref{thm1} and Lemma \ref{lem6}, $F$ and $G$ are defined.
Then $F$ and $G$ are (covariant) functors.
By Corollary \ref{cor5}, we have $G \circ F = \operatorname{id}_{\mathscr{C}}$.
Clearly, $F \circ G = \operatorname{id}_{\mathscr{D}}$ holds.
\end{proof}

Finally, we will describe all semiring automorphisms of tropical curves.
To do so, we start characterizing star-shaped tropical curves consisting of a finite number ($\ge 1$) of $[0, \infty]$:

\begin{lemma}
	\label{lem7}
Let $\Gamma$ be a tropical curve.
Then the following are equivalent:

$(1)$ there exists a model $(G, l)$ for $\Gamma$ such that $l(E(G)) = \{ \infty \}$; and

$(2)$ $\Gamma$ is a star-shaped tropical curve consisting of a finite number ($\ge 1$) of $[0, \infty]$.
\end{lemma}

\begin{proof}
Clearly $(2)$ implies $(1)$.

Assume that $(2)$ does not hold.
If the genus $g(\Gamma)$ is at least one, then, clearly, $(1)$ does not hold.
Let $g(\Gamma) = 0$.
If $\Gamma$ is a singleton, then $(1)$ does not hold.
Assume that $\Gamma$ is not a singleton.
If $\Gamma$ has no points of valence at least three, then there exists a unique number $a \in \boldsymbol{R}_{>0}$ such that $\Gamma$ is isometric to $[0, a]$.
Hence, in this case, $(1)$ does not hold.
Assume that $\Gamma$ has some points of valence at least three.
If such points are only one, then $\Gamma$ is star-shaped.
Since $(2)$ does not hold, the underlying graph $G_{\circ}$ of the canonical model $(G_{\circ}, l_{\circ})$ for $\Gamma$ has a leaf edge of a finite length.
Hence $(1)$ does not hold.
Assume that there exist at least two points $x, y$ of valence at least three.
Since there exist only a finite number of such points, we can assume that the only path $P$ from $x$ to $y$ contains no such points other than $x, y$.
Then $P$ is an edge of the $G_{\circ}$ and has a finite length.
This means that $(1)$ does not hold.
\end{proof}

\begin{cor}
	\label{cor6}
Let $\Gamma$ be a tropical curve.
Then the following are equivalent:

$(1)$ there exists a model $(G, l)$ for $\Gamma$ such that $l(E(G)) = \{ \infty \}$; and

$(2)$ $\operatorname{Aut}_{\boldsymbol{T}}(\operatorname{Rat}(\Gamma)) \not= \operatorname{Aut}_{\rm semiring}(\operatorname{Rat}(\Gamma))$.
\end{cor}

\begin{proof}
If $\Gamma$ is a singleton, then $\Gamma$ satisfies both $(1)$ and $(2)$ by Corollary \ref{cor4} since $\operatorname{Rat}(\Gamma) = \boldsymbol{T}$.

Assume that $\Gamma$ is not a singleton.
We will show that if $(1)$ does not hold, then $(2)$ does also not hold.
$\operatorname{Aut}_{\boldsymbol{T}}(\operatorname{Rat}(\Gamma))$ is a subgroup of $\operatorname{Aut}_{\rm semiring}(\operatorname{Rat}(\Gamma))$.
For $\psi \in \operatorname{Aut}_{\rm semiring}(\operatorname{Rat}(\Gamma))$, let $r := \psi(1)$ and $\varphi$ be the $r$-expansive map $\Gamma \to \Gamma$ corresponding to $\psi$.
Suppose that $\Gamma$ is homeomorphic to a circle $S^1$.
Let $x, y$ be antipodal points and $P_1, P_2$ be the distinct paths from $x$ to $y$.
Then these paths have the same length $\operatorname{dist}(x, y)$.
Hence $\varphi(P_1)$ and $\varphi(P_2)$ are distinct paths from $\varphi(x)$ to $\varphi(y)$ of length $r \cdot \operatorname{dist}(x, y)$.
Therefore we have $r = 1$ and thus $\psi \in \operatorname{Aut}_{\rm semiring}(\operatorname{Rat}(\Gamma))$.
If $\Gamma$ is not homeomorphic to a circle $S^1$, then $\Gamma$ has at least one point of valence other than two.
For the canonical model $(G_{\circ}, l_{\circ})$ for $\Gamma$, there exists an edge $e \in E(G_{\circ})$ suth that $l_{\circ}(e)$ is finite since $(1)$ does not hold.
By Lemma \ref{lem7}, the endpoints $v, w$ of $e$ (possibly $v = w$) have valences other than two.
Since $\varphi$ is a homeomorphism, it respects valences, and thus $\varphi(e) \in E(G_{\circ})$.
If $r > 1$, then $l_{\circ}(e) < l_{\circ}(\varphi(e))$.
By the same argument, we have the following sequences: 
\[
e, \varphi(e), \varphi^2(e), \ldots \in E(G_{\circ})
\]
 and 
\[
l_{\circ}(e) < l_{\circ}(\varphi(e)) < l_{\circ}(\varphi^2(e)) < \cdots.
\]
These mean that $\# E(G_{\circ}) = \infty$.
It is a contradiction.
When $r < 1$, we have a contradiction by a similar argument.
Hence $r$ must be one, and thus $(2)$ does not hold.

Assume that $(1)$ holds.
Then by Lemma \ref{lem7}, $\Gamma$ is a star-shaped tropical curve consisting of a finite number ($\ge 1$) of $[0, \infty]$.
Clearly, for any $r \in \boldsymbol{R}_{>0}$, there exists an $r$-expansive map $\Gamma \to \Gamma$ that is a dilatation.
By Lemma \ref{lem6}, this induces a semiring automorphism $\psi : \operatorname{Rat}(\Gamma) \to \operatorname{Rat}(\Gamma)$ such that $\psi(1) = r$.
Thus $\operatorname{Aut}_{\boldsymbol{T}}(\operatorname{Rat}(\Gamma))$ is a proper subgroup of $\operatorname{Aut}_{\rm semiring}(\operatorname{Rat}(\Gamma))$.
\end{proof}

\begin{prop}
	\label{prop2}
Let $\Gamma$ be a star-shaped tropical curve consisting of a finite number $n \ge 1$ of $[0, \infty]$.

$(1)$ If $n \not= 2$, then 
\[
\operatorname{Aut}(\Gamma) \cong S_n
\]
and 
\[
\operatorname{Aut}_{\rm semiring}(\operatorname{Rat}(\Gamma)) \cong \operatorname{Aut}(\Gamma) \times \boldsymbol{R}_{>0},
\]
where $S_n$ denotes the symmetric group of degree $n$ and $\boldsymbol{R}_{>0}$ is regarded as the group $\boldsymbol{R}_{>0}$ with the usual multiplication as its binary operation.

$(2)$ Let $n = 2$ and $\Gamma$ be identified with $[ -\infty, \infty ]$ and $\iota$ be the inversion with respect to $0$.
Then 
\[
\operatorname{Aut}(\Gamma) = \langle \{ \text{the inversion with respect to }x \,|\, x \in \Gamma \setminus \Gamma_{\infty} \} \rangle \cong \boldsymbol{R} \rtimes \langle \iota \rangle
\]
and 
\[
\operatorname{Aut}_{\rm semiring}(\operatorname{Rat}(\Gamma)) \cong \operatorname{Aut}(\Gamma) \rtimes \boldsymbol{R}_{>0},
\]
where $\langle \cdot \rangle$ denotes the group generated by $\cdot$, and $\boldsymbol{R}$ is regarded as the group $\boldsymbol{R}$ with the usual addition as its binary operation, and $\boldsymbol{R}_{>0}$ as the group $\boldsymbol{R}_{>0}$ with the usual multiplication as its binary operation.
\end{prop}

\begin{proof}
Let $R$ denote the set of all $r$-expansive maps $\Gamma \to \Gamma$ with all $r \in \boldsymbol{R}_{>0}$.
By Corollary \ref{cor2}, $R$ is isomorphic to $\operatorname{Aut}_{\rm semiring}(\operatorname{Rat}(\Gamma))$ as a group.

We shall prove $(1)$.
Let $n = 1$ and $\Gamma = [0, \infty]$.
For any $r > 0$ and any $r$-expansive map $\varphi : \Gamma \to \Gamma$, $\varphi(0)$ must be zero.
In fact, since $\varphi$ is a homeomorphism, $\varphi(0) = 0$ or $\infty$.
As $\operatorname{dist}(0, 1) = 1$, we have $\operatorname{dist}(\varphi(0), \varphi(1)) = r$.
As the point $\varphi(1)$ is in $(0, \infty)$, $\varphi(0)$ cannot be $\infty$.
Thus $R$ is isomorphic to $\boldsymbol{R}_{>0}$ as a group.
On the other hand, $\operatorname{Aut}(\Gamma)$ consists of only the identity of $\Gamma$.
Thus $\operatorname{Aut}_{\rm semiring}(\operatorname{Rat}(\Gamma))$ is isomorphic to $\operatorname{Aut}(\Gamma) \times \boldsymbol{R}_{>0}$ as a group.

Let $n \ge 3$.
There exists a unique point $x \in \Gamma$ of valence $n$.
Since any $r$-expansive map $\varphi : \Gamma \to \Gamma$ with any $r \in \boldsymbol{R}_{>0}$ is a homeomorphism, $\varphi(x) = x$ holds.
Thus $\operatorname{Aut}(\Gamma)$ is isomorphic to the symmetric group of degree $n$ as a group.
$\varphi$ defines the pair of an element of $\operatorname{Aut}(\Gamma)$ and a dilatation of $\Gamma$ with $r$ as its expansion factor.
Conversely, the pair of an element of $\operatorname{Aut}(\Gamma)$ and a dilatation of $\Gamma$ with expansion factor $r$ for each $r \in \boldsymbol{R}_{>0}$ defines an $r$-expansive map $\Gamma \to \Gamma$.
An element of $\operatorname{Aut}(\Gamma)$ and a dilatation of $\Gamma$ with expansion factor any $r \in \boldsymbol{R}_{>0}$ are commutative, and thus we have the conclusion.

We shall prove $(2)$.
Since $\Gamma = [- \infty, \infty]$, the first assertion is clear.
We show that $\operatorname{Aut}_{\rm semiring}(\operatorname{Rat}(\Gamma)) \cong \operatorname{Aut}(\Gamma) \rtimes \boldsymbol{R}_{>0}$.
Let $\varphi : \Gamma \to \Gamma$ be an $r$-expansive map with $r \in \boldsymbol{R}_{>0} \setminus \{ 1 \}$.
Then $\varphi$ has a unique fixed point $x$ other than $\pm \infty$.
In fact, it is easy to check that if $\varphi(\infty) = \infty$ and $\varphi(-\infty) = -\infty$, then $x = \frac{\varphi(0)}{1 - r}$, and that if $\varphi(\infty) = -\infty$ and $\varphi(-\infty) = \infty$, then $x = \frac{\varphi(0)}{1 + r}$.
In both cases, since $r \not= 1$, other finite points are not fixed.

We consider the group consisting of all dilatation $\theta_r : \Gamma \to \Gamma$ fixing $0$ with $r$ as its expansion factor for any $r \in \boldsymbol{R}_{>0}$, which is isomorphic to $\boldsymbol{R}_{>0}$ as a group.
If $\varphi(\infty) = \infty$ and $\varphi(-\infty) = -\infty$, then $\varphi = (\varphi_x \circ \operatorname{id}_{\Gamma}) \circ \theta_r$, where $\varphi_x$ denotes the translation $\Gamma \to \Gamma; y \mapsto y + x$.
If $\varphi(\infty) = -\infty$ and $\varphi(-\infty) = \infty$, then $\varphi = (\varphi_x \circ \iota) \circ \theta_r$.
Conversely, an element of $\operatorname{Aut}(\Gamma)\boldsymbol{R}_{>0}$ defines an element of $R$.
Thus $R = \operatorname{Aut}(\Gamma)\boldsymbol{R}_{>0}$.
Clearly both $\operatorname{Aut}(\Gamma)$ and $\boldsymbol{R}_{>0}$ are subgroups of $R$, and for $\psi \in R$ and $\varphi \in \operatorname{Aut}(\Gamma)$, $\psi \circ \varphi \circ \psi^{-1}$ is a $1$-expansive map, i.e., $\psi \circ \varphi \circ \psi^{-1}\in \operatorname{Aut}(\Gamma)$.
Hence $\operatorname{Aut}(\Gamma)$ is a normal subgroup of $R$.
Thus we have the conclusion.
\end{proof}

Note that in Proposition \ref{prop2}$(2)$, since there exist an element of $\operatorname{Aut}(\Gamma)$ and an element of $\boldsymbol{R}_{>0}$ which are not commutative, we have $\operatorname{Aut}_{\rm semiring}(\operatorname{Rat}(\Gamma)) \not\cong \operatorname{Aut}(\Gamma) \times \boldsymbol{R}_{>0}$.
For example, $((\varphi_1 \circ \operatorname{id}_{\Gamma}) \circ \theta_2)(0) = 1$ and $(\theta_2 \circ (\varphi_1 \circ \operatorname{id}_{\Gamma}))(0) = 2$, and thus $(\varphi_1 \circ \operatorname{id}_{\Gamma}) \circ \theta_2 \not= \theta_2 \circ (\varphi_1 \circ \operatorname{id}_{\Gamma})$.

\end{document}